\documentclass{amsart}

\newtheorem{theorem}{Theorem}[section]
\newtheorem{lemma}[theorem]{Lemma}
\newtheorem{corollary}[theorem]{Corollary}

\theoremstyle{definition}

\theoremstyle{remark}

\numberwithin{equation}{section}

\begin{document}

\title[Extension of the $H^{k}$ mean curvature flow]{On an extension of the $H^{k}$ mean curvature flow of closed convex
hypersurfaces}

\author{Yi Li}
\address{Department of Mathematics,
Shanghai Jiao Tong University, 800 Dongchuan Road, Min Hang District, 
Shanghai, 200240 China}

\email{yilicms@gmail.com}




\begin{abstract} In this paper we prove that the $H^{k}$ ($k$ is odd and larger
than $2$) mean curvature flow of
a closed convex hypersurface can be extended over the maximal time provided that the
total $L^{p}$ integral of the mean curvature is finite for some $p$.
\end{abstract}

\maketitle

\section{Introduction}

Let $M$ be a compact $n$-dimensional hypersurface without boundary, which is smoothly embedded into the $(n+1)$-dimensional Euclidean space $\mathbb{R}^{n+1}$ by the map
\begin{equation}
F_{0}: M\longrightarrow \mathbb{R}^{n+1}.\label{1.1}
\end{equation}
The $H^{k}$ mean curvature flow, an evolution equation of the mean curvature
$H(\cdot,t)$, is a smooth family of immersions $F(\cdot,t): M\to\mathbb{R}^{n+1}$ given by
\begin{equation}
\frac{\partial}{\partial t}F(\cdot,t)=-H^{k}(\cdot,t)\nu(\cdot,t), \ \ \ F(\cdot,0)
=F_{0}(\cdot),\label{1.2}
\end{equation}
where $k$ is a positive integer and $\nu(\cdot,t)$ denotes the outer unit
normal on $M_{t}:=F(M,t)$ at $F(\cdot,t)$.

When $k=1$ the equation (\ref{1.2})
is the usual mean curvature flow. Huisken \cite{Huisken} proved that the mean curvature
flow develops to singularities in finite time: Suppose that $T_{{\rm
max}}<\infty$ is the first singularity time for the mean curvature
flow. Then $\sup_{M_{t}}|A|(t)\to\infty$ as $t\to T_{{\rm max}}$. Recently,
Le-Sesum \cite{LS} and
Xu-Ye-Zhao \cite{XYZ} independently proved an extension theorem on the mean
curvature flow under some curvature conditions. A natural question
is whether we can extend general $H^{k}$ mean curvature flow
over the maximal time interval.

The short time existence of the $H^{k}$ mean curvature flow has been
established in \cite{S}, i.e., there is a maximal time interval
$[0,T_{\max})$, $T_{\max}<\infty$, on which the flow exists. In \cite{L}, we proved an extension theorem on the $H^{k}$ mean curvature flow under some curvature condition; that is,
the condition (b) in Theorem \ref{t1.1} holds and the second fundamental
form has a lower bound along the flow. In
this paper, we give another extension theorem of the
$H^{k}$ mean curvature flow for convex hypersurfaces.

\begin{theorem} \label{t1.1}Suppose that the integers $n$ and $k$ are greater than or equal
to $2$, $k$ is odd, and $n+1\geq k$. Suppose that $M$ is a compact $n$-dimensional
hypersurface without boundary, smoothly embedded into $\mathbb{R}^{n+1}$ by a smooth function $F_{0}$. Consider the $H^{k}$ mean curvature flow on $M$,
\begin{equation*}
\frac{\partial}{\partial t}F(\cdot, t)=-H^{k}(\cdot, t)\nu(\cdot,t), \ \ \ F(\cdot,0)=F_{0}(\cdot).
\end{equation*}
If
\begin{itemize}

\item[(a)] $H(\cdot)>0$ on $M$,

\item[(b)] for some $\alpha\geq n+k+1$,
\begin{equation*}
||H(\cdot,t)||_{L^{\alpha}(M\times[0,T_{\max}))}:=\left(\int^{T_{\max}}_{0}
\int_{M_{t}}|H(\cdot,t)|^{\alpha}_{g(t)}d\mu(t)dt\right)^{\frac{1}{\alpha}}<\infty,
\end{equation*}

\end{itemize}
then the flow can be extended over the time $T_{\max}$. Here $d\mu(t)$ denotes
the induced metric on $M_{t}$.
\end{theorem}

If the second fundamental form has a lower bound, i.e., $h_{ij}(t)\geq C g_{ij}(t)$,
then $H(t)\geq nC>0$ which satisfies condition (a). Therefore the above
theorem is a weak version of that in \cite{L}.

\section{Evolution equations for the $H^{k}$ mean curvature flow}

Let $g=\{g_{ij}\}$ be the induced metric on $M$ obtained by the pullback of
the standard metric $g_{{\mathbb{R}^{n+1}}}$ of $\mathbb{R}^{n+1}$. We denote by $A=\{h_{ij}\}$ the second fundamental form and $d\mu=\sqrt{{\rm det}(g_{ij})}dx^{1}\wedge\cdots
\wedge dx^{n}$ the volume form on $M$, respectively, where $x^{1},\cdots,x^{n}$
are local coordinates. The mean curvature can be expressed as
\begin{equation}
H=g^{ij}h_{ij}, \ \ \ g_{ij}=\left\langle\frac{\partial F}{\partial x^{i}},
\frac{\partial F}{\partial x^{j}}\right\rangle_{g_{\mathbb{R}^{n+1}}};\label{2.1}
\end{equation}
meanwhile the second fundamental forms are given by
\begin{equation}
h_{ij}=-\left\langle\nu,\frac{\partial^{2}F}{\partial x^{i}\partial x^{j}}\right\rangle_{g_{\mathbb{R}^{n+1}}}.\label{2.2}
\end{equation}

We write $g(t)=\{g_{ij}(t)\}, A(t)=\{h_{ij}(t)\}, \nu(t), H(t), d\mu(t),
\nabla_{t}$, and $\Delta_{t}$ the corresponding induced metric, second fundamental form, outer unit normal vector, mean curvature, volume form, induced Levi-Civita connection, and induced Laplacian operator at time $t$. The position coordinates are not explicitly written in
the above symbols if there is no confusion.

The following evolution equations are obvious.

\begin{lemma} \label{l2.1}For the $H^{k}$ mean curvature flow, we have
\begin{eqnarray*}
\frac{\partial}{\partial t}H(t)&=&kH^{k-1}(t)\Delta_{t}H(t)+H^{k}(t)
|A(t)|^{2}
+k(k-1)H^{k-2}(t)\left|\nabla_{t}H(t)\right|^{2},\\
\frac{\partial}{\partial t}|A(t)|^{2}&=&kH^{k-1}(t)\Delta_{t}|A(t)|^{2}-2kH^{k-1}(t)\left|\nabla_{t}A(t)
\right|^{2}+2kH^{k-1}(t)|A(t)|^{4}\\
&&+ \ 2k(k-1)H^{k-2}(t)|\nabla_{t} H(t)|^{2}.
\end{eqnarray*}
Here and henceforth, the norm $|\cdot|$ is respect to the induced metric $g(t)$.
\end{lemma}

\begin{corollary} \label{c2.2} Suppose that $\min_{M}H(0)>0$. If $k$ is odd and larger than $2$, then
\begin{equation}
H(t)\geq \min_{M}H(0)\label{2.3}
\end{equation}
along the $H^{k}$ mean curvature flow. In particular, $H(t)>0$ is preserved by the $H^{k}$ mean curvature flow.
\end{corollary}

\begin{proof} By Lemma \ref{l2.1}, we have
\begin{eqnarray*}
\frac{\partial}{\partial t}H(t)&=&kH^{k-1}(t)\Delta_{t}H(t)+H^{k}(t)|A(t)|^{2}
+k(k-1)H^{k-2}(t)\left|\nabla_{t}H(t)\right|^{2}\\
&=&kH^{k-1}(t)\Delta_{t}H(t)\\
&&+ \ \left(H^{k-1}(t)|A(t)|^{2}
+k(k-1)H^{k-3}(t)\left|\nabla_{t}H(t)\right|^{2}\right)H(t).
\end{eqnarray*}
Since $k\geq2$ and $k$ is odd, it follows that
\begin{equation*}
H^{k-1}(t)|A(t)|^{2}+k(k-1)H^{k-3}(t)\left|\nabla_{t}H(t)\right|^{2}
\end{equation*}
is nonnegative and then (\ref{2.3}) follows from the maximum principle.
\end{proof}

\begin{lemma} \label{l2.3}Suppose $k$ is odd and larger than $2$, and $H>0$. For the $H^{k}$ mean curvature flow and any positive integer $\ell$, we have
\begin{eqnarray*}
&&\left(\frac{\partial}{\partial t}-kH^{k-1}(t)\Delta_{t}\right)\left(\frac{|A(t)|^{2}}{H^{\ell+1}(t)}\right)\\
&=&\frac{k(\ell+1)}{k-1}\left\langle\nabla_{t} H^{k-1}(t),\nabla_{t}\left(\frac{|A(t)|^{2}}{H^{\ell+1}(t)}\right)\right\rangle
\\
&&- \ \frac{2k}{H^{\ell+4-k}(t)}\left[\left(H(t)\nabla_{t}A(t)
-\frac{\ell+1}{2}A(t)\nabla_{t}H(t)\right)\right]^{2}+\frac{2k(k-1)}{H^{\ell+3-k}(t)}\left|\nabla_{t}H(t)\right|^{2}\\
&&+ \ \frac{2k-\ell-1}{H^{\ell+2-k}(t)}|A(t)|^{4}
-\frac{k(\ell+1)(2k-\ell-1)}{2H^{\ell+4-k}(t)}|A(t)|^{2}\left|\nabla_{t}H(t)
\right|^{2}.
\end{eqnarray*}
\end{lemma}

\begin{proof} In the following computation, we will always omit time $t$ and write $\partial/\partial t$ as $\partial_{t}$. Then
\begin{equation*}
\partial_{t} H=kH^{k-1}\Delta H+H^{k}|A|^{2}+k(k-1)H^{k-2}|\nabla H|^{2}.
\end{equation*}
By Corollary \ref{c2.2}, $H(t)>0$ along the $H^{k}$ mean curvature flow so
that $|H(t)|^{i}=H^{i}(t)$ for each positive integer $i$. For any positive integer $\ell$, we have
\begin{eqnarray*}
\partial_{t}|H|^{\ell+1}&=&(\ell+1)H^{\ell}\partial_{t}H\\
&=&(\ell+1)H^{\ell}\left(kH^{k-1}\Delta H+H^{k}|A|^{2}+k(k-1)H^{k-2}|\nabla H|^{2}\right)\\
&=&k(\ell+1)H^{k+\ell-1}\Delta H+(\ell+1)H^{k+\ell}|A|^{2}\\
&&+ \ k(k-1)(\ell+1)H^{k+\ell-2}|\nabla H|^{2},\\
\Delta|H|^{\ell+1}&=&\Delta H^{\ell+1} \ \ = \ \ (\ell+1)\nabla\left(H^{\ell}\nabla H\right)\\
&=&(\ell+1)\left(\ell H^{\ell-1}|\nabla H|^{2}+H^{\ell}\Delta H\right)\\
&=&(\ell+1)H^{\ell}\Delta H+\ell(\ell+1)H^{\ell-1}|\nabla H|^{2}.
\end{eqnarray*}
Therefore
\begin{eqnarray}
\partial_{t}H^{\ell+1}&=&k H^{k-1}\Delta H^{\ell+1}-k\ell (\ell+1)
H^{k+\ell-2}|\nabla H|^{2}\nonumber\\
&&+ \ (\ell+1)H^{k+\ell}|A|^{2}+k(k-1)(\ell+1)H^{k+\ell-2}|\nabla H|^{2}\nonumber\\
&=&k H^{k-1}\Delta H^{\ell+1}+(\ell+1)H^{k+\ell}|A|^{2}\label{2.4}\\
&&+ \ k(k-\ell-1)(\ell+1)H^{k+\ell-2}|\nabla H|^{2}.\nonumber
\end{eqnarray}
Recall from Lemma \ref{l2.1} that
\begin{equation*}
\partial_{t}|A|^{2}=kH^{k-1}\Delta|A|^{2}-2kH^{k-1}|\nabla A|^{2}+2kH^{k-1}|A|^{4}
+2k(k-1)H^{k-2}|\nabla H|^{2}.
\end{equation*}
Calculate, using (\ref{2.4}),
\begin{eqnarray*}
&&\partial_{t}\left(\frac{|A|^{2}}{|H|^{\ell+1}}\right) \ \ = \ \ \frac{\partial_{t}|A|^{2}}{|H|^{\ell+1}}-\frac{|A|^{2}}{|H|^{2\ell+2}}\partial_{t}|H|^{\ell+1}\\
&=&\frac{kH^{k-1}\Delta |A|^{2}-2kH^{k-1}|\nabla A|^{2}+2kH^{k-1}|A|^{4}+2k(k-1)H^{k-2}|\nabla H|^{2}}{H^{\ell+1}}\\
&&- \ \frac{|A|^{2}\left[kH^{k-1}\Delta H^{\ell+1}+(\ell+1)H^{k+\ell}|A|^{2}
+k(k-\ell-1)(\ell+1)H^{k+\ell-2}|\nabla H|^{2}\right]}{H^{2\ell+2}}\\
&=&kH^{k-1}\frac{1}{H^{\ell+1}}\Delta|A|^{2}-\frac{2k}{H^{\ell+2-k}}|\nabla A|^{2}
+\frac{2k}{H^{\ell+2-k}}|A|^{4}+\frac{2k(k-1)}{H^{\ell+3-k}}|\nabla H|^{2}\\
&&- \ \frac{k|A|^{2}}{H^{2\ell+3-k}}\Delta H^{\ell+1}-\frac{\ell+1}{H^{\ell+2-k}}
|A|^{4}-\frac{k(k-\ell-1)(\ell+1)}{H^{\ell+4-k}}|A|^{2}|\nabla H|^{2},
\end{eqnarray*}
and
\begin{eqnarray*}
\Delta\left(\frac{|A|^{2}}{H^{\ell+1}}\right)&=&\frac{1}{H^{\ell+1}}\Delta|A|^{2}
+\Delta\left(\frac{1}{H^{\ell+1}}\right)|A|^{2}+2\left\langle\nabla|A|^{2},
\nabla\left(\frac{1}{H^{\ell+1}}\right)\right\rangle,\\
\nabla\left(\frac{1}{H^{\ell+1}}\right)&=&\frac{-(\ell+1)
H^{\ell}\nabla H}{H^{2\ell+2}} \ \ = \ \ \frac{-(\ell+1)\nabla H}{H^{\ell+2}},\\
\Delta\left(\frac{1}{H^{\ell+1}}\right)&=&\nabla\left(\frac{-(\ell+1)
\nabla H}{H^{\ell+2}}\right)\\
&=&-(\ell+1)\frac{H^{\ell+2}\Delta H -\nabla H(\ell+2)H^{\ell+1}\nabla H}{H^{2\ell+4}}\\
&=&-(\ell+1)\left[\frac{\Delta H}{H^{\ell+2}}
-(\ell+2)\frac{|\nabla H|^{2}}{H^{\ell+3}}\right],\\
\Delta H^{\ell+1}&=&\nabla\left[(\ell+1)H^{\ell}\nabla H\right] \ \ = \ \
(\ell+1)\left[\ell H^{\ell-1}|\nabla H|^{2}+H^{\ell}\Delta H\right]\\
&=&\ell(\ell+1)H^{\ell-1}|\nabla H|^{2}+(\ell+1)H^{\ell}\Delta H.
\end{eqnarray*}
Combining with all of them yields
\begin{eqnarray*}
&&\left(\partial_{t}-kH^{k-1}\Delta\right)\left(\frac{|A|^{2}}{H^{\ell+1}}\right)\\
&=&kH^{k-\ell-2}\Delta |A|^{2}-\frac{2k}{H^{\ell+2-k}}|\nabla A|^{2}
+\frac{2k}{H^{\ell+2-k}}|A|^{4}
+\frac{2k(k-1)}{H^{\ell+3-k}}|\nabla H|^{2}\\
&&- \ \frac{k|A|^{2}}{H^{2\ell+3-k}}\left[\ell(\ell+1)H^{\ell-1}|\nabla H|^{2}
+(\ell+1)H^{\ell}\Delta H\right]-\frac{\ell+1}{H^{\ell+2-k}}|A|^{4}\\
&&- \ \frac{k(k-\ell-1)(\ell+1)|A|^{2}}{H^{\ell-k+4}}|\nabla H|^{2}\\
&&- \ kH^{k-1}\left[\frac{1}{H^{\ell+1}}\Delta|A|^{2}-(\ell+1)\frac{|A|^{2}\Delta H}{H^{\ell+2}}
+(\ell+1)(\ell+2)\frac{|A|^{2}|\nabla H|^{2}}{H^{\ell+3}}\right]\\
&&- \ 2kH^{k-1}\left\langle\nabla|A|^{2},\nabla\left(\frac{1}{H^{\ell+1}}\right)\right\rangle\\
&=&-\frac{2k}{H^{\ell+2-k}}|\nabla A|^{2}+\left(\frac{2k}{H^{\ell+2-k}}
-\frac{\ell+1}{H^{\ell+2-k}}\right)|A|^{4}+\frac{2k(k-1)}{H^{\ell+3-k}}|\nabla H|^{2}\\
&&- \ \frac{k(\ell+1)(k+\ell+1)|A|^{2}|\nabla H|^{2}}{H^{\ell+4-k}}-2kH^{k-1}
\left\langle\nabla|A|^{2},\nabla\left(\frac{1}{H^{\ell+1}}\right)\right\rangle.
\end{eqnarray*}
On the other hand,
\begin{eqnarray*}
\left\langle\nabla|A|^{2},\nabla\left(\frac{1}{H^{\ell+1}}\right)\right\rangle
&=&2\left\langle\nabla A \cdot A,\frac{-(\ell+1)H^{\ell}\nabla H}{H^{2\ell+2}}
\right\rangle\\
&=&\frac{-2(\ell+1)}{H^{\ell+3}}\langle H\nabla A\cdot A, \nabla H
\rangle.
\end{eqnarray*}
Thus, we conclude that
\begin{eqnarray*}
&&\left(\partial_{t}-kH^{k-1}\Delta\right)\left(\frac{|A|^{2}}{H^{\ell+1}}\right)\\
&=&-\frac{2k}{H^{\ell+2-k}}|\nabla A|^{2}+\frac{2k-\ell-1}{H^{\ell+2-k}}|A|^{4}
+\frac{2k(k-1)}{H^{\ell+3-k}}|\nabla H|^{2}\\
&&- \ \frac{k(\ell+1)(k+\ell+1)|A|^{2}|\nabla H|^{2}}{H^{\ell+4-k}}
+\frac{4k(\ell+1)}{H^{\ell+4-k}}\langle H\nabla A\cdot A, \nabla H
\rangle.
\end{eqnarray*}
Consider the function
\begin{equation*}
f:=\frac{-2k}{H^{\ell+2-k}}|\nabla A|^{2}
-\frac{k(\ell+1)(k+\ell+1)|A|^{2}|\nabla H|^{2}}{H^{\ell+4-k}}
+\frac{4k(\ell+1)}{H^{\ell+4-k}}\langle H\nabla A\cdot A, \nabla H
\rangle.
\end{equation*}
Since
\begin{eqnarray*}
\frac{2k(\ell+1)}{H^{\ell+4-k}}\langle H\nabla A\cdot A, \nabla H
\rangle&=&
\frac{k(\ell+1)}{H^{\ell+3-k}}\left\langle\nabla|A|^{2},\nabla H\right\rangle,\\
\nabla\left(\frac{|A|^{2}}{H^{\ell+1}}\right)&=&\frac{\nabla|A|^{2}}{H^{\ell+1}}
-\frac{(\ell+1)|A|^{2}\nabla H}{H^{\ell+2}},
\end{eqnarray*}
it follows that
\begin{eqnarray*}
\frac{2k(\ell+1)}{H^{\ell+4-k}}\langle H\nabla A\cdot A, \nabla H
\rangle&=&
\frac{k(\ell+1)}{H^{2-k}}\nabla H\left[\nabla\left(\frac{|A|^{2}}{H^{\ell+1}}\right)
+\frac{(\ell+1)|A|^{2}\nabla H}{H^{\ell+2}}\right]\\
&=&\frac{k(\ell+1)}{k-1}\left\langle\nabla H^{k-1},\nabla\left(\frac{|A|^{2}}{H^{\ell+1}}\right)\right\rangle\\
&&+ \ \frac{k(\ell+1)^{2}}{H^{\ell+4-k}}|A|^{2}|\nabla H|^{2}.
\end{eqnarray*}
Consequently,
\begin{eqnarray*}
f&=&\frac{-2k}{H^{\ell+2-k}}|\nabla A|^{2}-\frac{k^{2}(\ell+1)}{H^{\ell+4-k}}|A|^{2}|\nabla H|^{2}\\
&&+ \ \frac{k(\ell+1)}{k-1}\left\langle\nabla H^{k-1},\nabla\left(\frac{|A|^{2}}{H^{\ell+1}}\right)\right\rangle
+\frac{2k(\ell+1)}{H^{\ell+4-k}}\langle H\nabla A\cdot A, \nabla H
\rangle\\
&=&\frac{-2k}{H^{\ell+4-k}}\left[\left(H\nabla A-\frac{\ell+1}{2}A\cdot\nabla H\right)^{2}\right]\\
&&- \ \frac{2k(\ell+1)(2k-\ell-1)}{4H^{\ell+4-k}}|A|^{2}|\nabla H|^{2}+\frac{k(\ell+1)}{k-1}\left\langle\nabla H^{k-1},\nabla\left(\frac{|A|^{2}}{H^{\ell+1}}\right)\right\rangle.
\end{eqnarray*}
Finally, we complete the proof.
\end{proof}

\begin{corollary} \label{c2.4}Suppose $k$ is odd and larger than $2$, and $H>0$. For the $H^{k}$ mean curvature flow, we have
\begin{eqnarray*}
&&\left(\frac{\partial}{\partial t}-kH^{k-1}(t)\Delta_{t}\right)\left(\frac{|A(t)|^{2}}{H^{2k}(t)}\right)\\
&=&\frac{2k^{2}}{k-1}\left\langle\nabla_{t}H^{k-1}(t),\nabla_{t}
\left(\frac{|A(t)|^{2}}{H^{2k}(t)}\right)\right\rangle+\frac{2k(k-1)}{H^{k+2}(t)}
\left|\nabla_{t}H(t)\right|^{2}\\
&&- \ \frac{2k}{H^{k+3}(t)}\left[H(t)\cdot\nabla_{t}A(t)
-kA(t)\cdot\nabla_{t}H(t)\right]^{2}.
\end{eqnarray*}
\end{corollary}

\section{Proof of the main theorem}\label{section3}

In this section we give a proof of theorem \ref{t1.1}. For any positive constant $C_{0}$, consider the quantity
\begin{equation}
Q(t):=\frac{|A(t)|^{2}}{H^{2k}(t)}+C_{0}H^{\ell+1}(t),\label{3.1}
\end{equation}
where the integer $\ell$ is determined later. By (\ref{2.4}) and
Corollary \ref{c2.4}, we have
\begin{eqnarray*}
&&\left(\frac{\partial}{\partial t}-kH^{k-1}(t)\Delta_{t}\right)Q(t)\\
&\leq&\frac{2k^{2}}{k-1}\left\langle\nabla_{t} H^{k-1}(t),\nabla_{t} Q(t)
-C_{0}\nabla_{t} H^{\ell+1}(t)\right\rangle
+\frac{2k(k-1)}{H^{k+2}(t)}\left|\nabla_{t}H(t)\right|^{2}\\
&&+ \ C_{0}\left[(\ell+1)H^{k+\ell}(t)
|A(t)|^{2}+k(k-\ell-1)(\ell+1)H^{k+\ell-2}(t)\left|\nabla_{t}H(t)\right|^{2}\right]\\
&=&\frac{2k^{2}}{k-1}\left\langle\nabla_{t}H^{k-1}(t),\nabla_{t}Q(t)
\right\rangle-\frac{2k^{2}}{k-1}C_{0}(k-1)(\ell+1)H^{k+\ell-2}(t)\left|\nabla_{t}H(t)
\right|^{2}\\
&&+ \ \frac{2k(k-1)}{H^{k+2}(t)}\left|\nabla_{t}H(t)\right|^{2}
+C_{0}k(k-\ell-1)(\ell+1)H^{k+\ell-2}(t)\left|\nabla_{t}H(t)\right|^{2}\\
&&+ \ C_{0}(\ell+1)H^{k+\ell}(t)\left[Q(t)-C_{0}H^{\ell+1}(t)\right]H^{2k}(t)\\
&=&\frac{2k^{2}}{k-1}\left\langle\nabla_{t}H^{k-1}(t),\nabla_{t}Q(t)
\right\rangle\\
&&+ \ \left|\nabla_{t}H(t)\right|^{2}
\left[\frac{2k(k-1)}{H^{k+2}(t)}-C_{0}k(\ell+1)(k+\ell+1)H^{k+\ell-2}(t)
\right]\\
&&+ \ C_{0}(\ell+1)H^{3k+\ell}(t)Q(t)-C^{2}_{0}(\ell+1)H^{3k+2\ell+1}(t).
\end{eqnarray*}
Now we choose $\ell$ so that the following constraints
\begin{equation*}
\ell+1\leq0, \ \ \ k+\ell+1\leq0, \ \ 3k+2\ell+1\geq0
\end{equation*}
are satisfied; that is
\begin{equation}
-\frac{1}{2}-\frac{3}{2}k\leq\ell\leq-1-k.\label{3.2}
\end{equation}
In particular, we can take
\begin{equation}
\ell:=-2-k.\label{3.3}
\end{equation}
By our assumption on $k$, we have $k\geq3$ and hence (\ref{3.3}) implies
(\ref{3.2}). Plugging (\ref{3.3}) into the above inequality yields
\begin{eqnarray}
&&\left(\frac{\partial}{\partial t}-kH^{k-1}(t)\Delta_{t}\right)Q(t)\nonumber\\
&\leq&\frac{2k^{2}}{k-1}\left\langle\nabla H^{k-1}(t),\nabla_{t}Q(t)
\right\rangle+\left|\nabla_{t}H(t)\right|^{2}
\left[\frac{2k(k-1)}{H^{k+2}(t)}-\frac{C_{0}k(k+1)}{H^{4}(t)}\right]\label{3.4}\\
&&- \ C_{0}(1+k)H^{2k-2}(t)Q(t)+C^{2}_{0}(1+k)H^{k-3}(t).\nonumber
\end{eqnarray}
Choosing
\begin{equation}
C_{0}:=\frac{2(k-1)}{k+1}H^{2-k}_{\min}>0\label{3.5}
\end{equation}
where $H_{\min}:=\min_{M}H=\min_{M}H(0)$, we arrive at
\begin{equation*}
\frac{2k(k-1)}{C_{0}k(k+1)}\leq H^{k-2}_{\min}\leq H^{k-2}(0)\leq H^{k-2}(t)
\end{equation*}
according to (\ref{2.3}). Consequently,
\begin{eqnarray}
\left(\frac{\partial}{\partial t}-kH^{k-1}(t)\Delta_{t}\right)Q(t)
&\leq&\frac{2k^{2}}{k-1}\left\langle\nabla_{t}H^{k-1}(t),\nabla_{t}Q(t)
\right\rangle\nonumber\\
&&- \ C_{1}H^{2k-2}(t)Q(t)+C_{2}H^{k-3}(t),\label{3.6}
\end{eqnarray}
for $C_{1}:=C_{0}(1+k)$ and $C_{2}:=C^{2}_{0}(1+k)$.

\begin{lemma}\label{l3.1} If the solution can not be extended over $T_{\max}$, then $H(t)$ is unbounded.
\end{lemma}

\begin{proof} By the assumption, we know that $|A(t)|$ is unbounded as
$t\to T_{\max}$. We now claim that $H(t)$ is also unbounded. Otherwise, $0<H_{\min}\leq H(t)\leq C$ for some uniform constant $C$. If we set
\begin{equation*}
C_{3}:=C_{1}H^{2k-2}_{\min}, \ \ \ C_{4}:=C_{2}C^{k-3},
\end{equation*}
then (\ref{3.6}) implies that
\begin{equation}
\left(\frac{\partial}{\partial t}-kH^{k-1}(t)\Delta_{t}\right)Q(t)
\leq\frac{2k^{2}}{k-1}\left\langle\nabla_{t}H^{k-1}(t),\nabla_{t}Q(t)
\right\rangle-C_{3}Q(t)+C_{4}.\label{3.7}
\end{equation}
By the maximum principle, we have
\begin{equation}
\mathcal{Q}'(t)\leq-C_{3}\mathcal{Q}(t)+C_{4}\label{3.8}
\end{equation}
where
\begin{equation*}
\mathcal{Q}(t):=\max_{M}Q(t).
\end{equation*}
Solving (\ref{3.8}) we find that
\begin{equation*}
\mathcal{Q}(t)\leq\frac{C_{4}}{C_{3}}+\left(\mathcal{Q}(0)-\frac{C_{4}}{C_{3}}
\right)e^{-C_{3}t}.
\end{equation*}
Thus $Q(t)\leq C_{5}$ for some uniform constant $C_{5}$. By the definition (\ref{3.1}) and the assumption $H(t)\leq C$, we conclude that $|A(t)|\leq C_{6}$ for some
uniform constant $C_{6}$, which is a contradiction.
\end{proof}

The rest proof is similar to \cite{L, XYZ}. Using Lemma \ref{l3.1} and the argument
in \cite{L} or in \cite{XYZ}, we get a contradiction and then the solution of the $H^{k}$ mean curvature flow can be extended over $T_{\max}$.

\bibliographystyle{amsplain}

\end{document}